\documentclass[11pt,twoside]{article}
\usepackage{fullpage}

\usepackage[T1]{fontenc}      
\usepackage{microtype}       
\usepackage{comment}
\usepackage{textcomp}
\usepackage{soul}            
\setuldepth{Berlin}
\usepackage{enumitem}        

\usepackage{amsmath, amsthm, amssymb, amsfonts}
\usepackage{mathtools}       
\usepackage{bm}              
\usepackage{bbm}             
\usepackage{nicefrac}
\usepackage{thm-restate}     

\usepackage{graphicx}
\usepackage{subfigure}       
\usepackage{epsf}
\usepackage{xcolor}
\usepackage{booktabs}        
\usepackage{makecell}        
\usepackage{adjustbox}
\usepackage{tikz}
\usetikzlibrary{positioning,decorations.pathreplacing,arrows.meta,calc,shapes.geometric, shapes.misc}

\usepackage[ruled, vlined, lined, commentsnumbered]{algorithm2e}

\usepackage{natbib}
\setcitestyle{round}

\usepackage[
    colorlinks,
    citecolor={blue!60!black},
    urlcolor={blue!70!black},
    linkcolor={red!60!black},
    breaklinks,
    hypertexnames=false
]{hyperref}

\usepackage[capitalize,nameinlink]{cleveref}[0.19]
\usepackage{etoc} 

\newtheorem{theorem}{Theorem}
\newtheorem{example}{Example}

\newtheorem{proposition}{Proposition}
\newtheorem{lemma}{Lemma}
\newtheorem{corollary}{Corollary}
\newtheorem{definition}{Definition}
\newtheorem{remark}{Remark}

\newlength{\widebarargwidth}
\newlength{\widebarargheight}
\newlength{\widebarargdepth}

\makeatletter
\long\def\@makecaption#1#2{
       \vskip 0.8ex
       \setbox\@tempboxa\hbox{\small {\bf #1:} #2}
       \parindent 1.5em 
       \dimen0=\hsize
       \advance\dimen0 by -3em
       \ifdim \wd\@tempboxa >\dimen0
               \hbox to \hsize{
                       \parindent 0em
                       \hfil 
                       \parbox{\dimen0}{\def\baselinestretch{0.96}\small
                               {\bf #1.} #2
                               } 
                       \hfil}
       \else \hbox to \hsize{\hfil \box\@tempboxa \hfil}
       \fi
       }
\makeatother

\newenvironment{fminipage}%
  {\begin{Sbox}\begin{minipage}}%
  {\end{minipage}\end{Sbox}\fbox{\TheSbox}}

\DeclareMathOperator*{\argmin}{argmin}

\crefname{section}{section}{sections}
\crefname{subsection}{subsection}{subsections}
\Crefname{section}{Section}{Sections}
\Crefname{subsection}{Subsection}{Subsections}
\Crefname{figure}{Figure}{Figures}

\crefformat{equation}{\textup{#2(#1)#3}}
\crefrangeformat{equation}{\textup{#3(#1)#4--#5(#2)#6}}
\crefmultiformat{equation}{\textup{#2(#1)#3}}{ and \textup{#2(#1)#3}}{, \textup{#2(#1)#3}}{, and \textup{#2(#1)#3}}

\Crefformat{equation}{#2Equation~\textup{(#1)}#3}
\Crefrangeformat{equation}{Equations~\textup{#3(#1)#4--#5(#2)#6}}

\AddToHook{env/lemma/begin}{\crefalias{theorem}{lemma}}
\AddToHook{env/corollary/begin}{\crefalias{theorem}{corollary}}
\AddToHook{env/proposition/begin}{\crefalias{theorem}{proposition}}
\AddToHook{env/definition/begin}{\crefalias{theorem}{definition}}
\crefdefaultlabelformat{#2\textup{#1}#3}

\begin{document}
\begin{center}

  {\bf{\LARGE{On the $\Omega(n)$ SDP relaxation gap for Submodular Box-Constrained  Quadratic Programming}}}

\vspace*{.2in}

{\large{
\begin{tabular}{ccc}
Junyu Zhang$^{*}$, Guanyi Wang$^{\dag}$
\end{tabular}
}}
\vspace*{.2in}

\begin{tabular}{c}
$^{*,\dag}$Department of Industrial Systems Engineering and Management, \\
National University of Singapore\\
$^{*}$junyuz@nus.edu.sg, \quad $^{\dag}$guanyi.w@nus.edu.sg
\end{tabular}

\vspace*{.2in}

\today

\vspace*{.2in}

\begin{abstract}
Continuous submodular optimization is an important class of nonconvex problems with global guarantees, for which the submodular box-constrained quadratic programming (BCQP) forms a fundamental subclass. It arises naturally as quadratic subproblems of general twice differentiable submodular problems and also has direct business and management applications such as pricing. This paper studies the worst-case relaxation gap of semidefinite programming (SDP) relaxations with instance-independent finite family of valid linear cuts for submodular BCQP. We show that in dimension $n\geq4$, for any SDP relaxation formulated with a finite family of linear cuts, there always exists a submodular BCQP instance with a strictly positive relaxation gap. By restricting to the normalized BCQP instances that removes the problem's scaling factor, we further quantify this limitation with explicit lower bounds. In general dimension $n\geq4$, when the SDP relaxation is formulated with only Boolean-quadric-polytope (BQP) valid cuts, including the commonly used McCormick and triangle inequalities, we establish an $\Omega(n)$ relaxation gap lower bound. For an arbitrary family of $m$ valid linear cuts, we establish an $\Omega(n/m^2)$ lower bound. These results capture the fundamental limitations of SDP relaxations, especially SDP with BQP valid cuts, and provide guidance on what additional constraints may be needed for more effective convex approximations of submodular BCQP.

\vspace{2ex} 
\noindent\textbf{Keywords:} submodular quadratic programming;  semidefinite relaxation;  relaxation gap;  lower bound
\end{abstract}
\end{center}

\section{Introduction}\label{sec:Intro}
Continuous submodular optimization \citep{bian2017continuous,bach2018efficient,bach2019submodular,axelrod2020near} is a natural generalization of the classical discrete submodular optimization \citep{lovasz1983submodular}, which may lead to global optimality guarantees regardless of nonconvexity \citep{bach2019submodular}. 
In this paper, we focus on an important subclass of continuous submodular optimization with quadratic objective function: 
\begin{equation}
\label{eq:boxqp}
\min_{x}\; x^\top Qx+c^\top x,
\quad\mbox{s.t.}\quad x\in[0,1]^n, 
\end{equation} 
where $Q_{ij}\leq0, \forall i\neq j$. And we call it submodular box-constrained quadratic programming (BCQP). This form of problem naturally arises from the quadratic subproblems of the when solving general twice differentiable submodular problems \citep{bach2018efficient}, and itself also has direct applications in the distributionally robust optimization problems from pricing \citep{burer2025semidefinite}. An existing randomized approximation algorithm \citep{axelrod2020near} for general continuous submodular optimization exhibits an 
$\widetilde{O}(nL^6/\epsilon^6)$ complexity, where $L$ is the function Lipschitz constant under 
$\ell_\infty$-norm. For the set of normalized submodular BCQP instances that removes a problem's own scaling issue:
\begin{equation} 
\label{eq:normal-QP}
\mathcal{N}_\infty^n:= \left\{(Q,c)\in\mathbb{S}^{n}\times\mathbb{R}^n:
\begin{array}{lll}
    \|Q\|_\infty\leq1,\,\,\|c\|_\infty\leq1\\
    Q_{ij} \leq 0, \forall ~ i\neq j\in[n]
\end{array} \right\},
\end{equation} 
where $\|Q\|_{\infty} := \max_{i,j} |Q_{i,j}|$,  we have $L = O(n^2)$, leading to a worst-case complexity of $\widetilde{O}(n^{13}/\epsilon^6)$. Therefore, more efficient QP-specific algorithms such as semidefinite programming (SDP) relaxations still remain interesting. Indeed, in \cite{bach2018efficient}, their quadratic subproblems are approximated by the following SDP:
$$\min_{x,X}\,\,Q\bullet X+c^\top x\quad\mathrm{s.t.}\quad X\leq \mathbf{1}x^\top,\,\,\, 
\left(\begin{aligned}
    1 \,\,\,& \,\,\,\,x^\top_B\\[-2mm]
    \,\,x_B & \,\,X_{BB}
\end{aligned}\right)\succeq 0,\,\,\,
\forall B=\{i,j\}\subseteq[n], i\neq j,$$
where $x_B$ and $X_{BB}$ follows the standard subvector and submatrix notations. However, this relaxation is not tight. A counterexample in dimension $n=3$ was presented by \cite{bach2018efficient}. By modifying the pair-wise positive semidefinite (PSD) constraint to the PSD constraint for all variables:
$$\min_{x,X}\,\,Q\bullet X+c^\top x\quad\mathrm{s.t.}\quad X\leq \mathbf{1}x^\top, \quad 
\bigg(\begin{aligned}
    1 \,& \,\,x^\top\\[-2mm]
    x & \,\,X
\end{aligned}\bigg)\succeq 0,$$
\cite{burer2025semidefinite} proved the tightness of the above SDP relaxation for $n\leq3$. Yet unfortunately, in their recently updated manuscript, a counterexample in dimension $n=4$ is provided. 

Given existing results, we study the approximation properties of a SDP relaxation of the following format: 
\begin{equation}
\label{eq:sdp-strengthen}
\min_{(x,X)\in\mathcal{S}_n}\,Q\bullet X+c^{\top} x\,\,\,\,\,\mathrm{s.t.}\,\,\,\,\,   
\ell_i(x,X)\geq0,\,\,\, i\in[m],
\end{equation}
where 
$$\mathcal{S}_n:=\Big\{(x,X)\in\mathbb{R}^n\times\mathbb{S}^n: \Big(\begin{aligned}
    1 \,& \,x^{\!\top}\\[-2 mm]
    x & \,X
\end{aligned}\Big)\succeq 0\Big\}$$
denote the set of lifted variables satisfying positive semidefinite constraint, and each $\ell_i$ is a general linear inequality given formally by Definition~\ref{definition:Valid-Cuts}. Hence would like to ask the following questions:

\begin{quote}
\centering
    \emph{Given any $n\geq 4$, is there a finite set of linear cuts $\ell_1,\cdots,\ell_m$ s.t. the SDP  relaxation \eqref{eq:sdp-strengthen}\\ is tight for all $(Q,c)\in\mathcal{N}_\infty^n$? If not, how does the relaxation gap scale w.r.t. $n$ and $m$?\,\,\,\,\, }
\end{quote}
To answer these questions, we start by geometrically recasting the universal tightness of an SDP relaxation as an equivalent convex set inclusion problem in the lifted space. Define the feasible region of the SDP relaxation \eqref{eq:sdp-strengthen} as 
$$\mathcal{R}_n = \left\{(x,X)\in\mathcal{S}_n: \ell_j(x,X)\ge0,  \forall  j \in [m]\right\},$$
the convex hull of the set of all rank-one lifted variables as 
$$\mathcal{M}_n:=\mathrm{conv}\left\{(x,xx^\top):x\in[0,1]^n\right\},$$ and a cone
\begin{equation} 
\label{eq:D-n}
\mathcal{D}_n:= \left\{\!(0,H)\in\mathbb{R}^{n}\times\mathbb{S}^n:
\begin{array}{lll} 
    H_{ii} = 0,\quad\,\,\,\,\, \forall ~ i \in [n] \\
    H_{ij} \leq 0, \forall ~ i\neq j\in[n]
\end{array} \!\right\}.
\end{equation}
Then we show the following equivalence: 
$$\mbox{Relaxation }\eqref{eq:sdp-strengthen}\mbox{ is tight on }\mathcal{N}_\infty^n \,\,\,\Longleftrightarrow\,\,\, \mathcal{R}_n\subseteq \mathcal{M}_n+\mathcal{D}_n.$$
This immediately suggests that, given an arbitrary $\mathcal{R}_n$, if there exists  
$(\bar x,\bar X)\in
\mathcal R_n\setminus(\mathcal M_n+\mathcal D_n)$, then the
separating hyperplane between $(\bar x,\bar X)$ and $\mathcal M_n+\mathcal D_n$ will lead to an instance $(\bar Q,\bar c)\in\mathcal{N}_\infty^n$ for which the SDP relaxation \eqref{eq:sdp-strengthen} is not tight.

Based on this observation, we first fix $\mathcal{R}_4$ as the SDP feasible region with McCormick inequalities. Then, given randomly constructed boundary points of $\mathcal{R}_4$, we provide a semi-infinite linear program formulation for finding a submodular linear separator in $(Q,c)$, followed by an alternating optimization scheme for further refinement. Interestingly, the iterates exhibit a highly structured pattern, which leads us to conjecture and eventually identify a simple parameterized family:
\begin{equation}
\label{eq:bad-family}
Q(a,b)=
\begin{pmatrix}
a^2 & -a & 0 & 0 \\
-a & 1 & -1 & 0 \\
0 & -1 & 1 & -b \\
0 & 0 & -b & \frac{1}{4}
\end{pmatrix}, \quad 
c(a,b)=
\begin{pmatrix}
a\\1\\2b-1\\0
\end{pmatrix}.
\end{equation}
All instances with $a,b\in[1/2,1]$ share the same QP optimal value $0$, while the corresponding (McCormick) SDP relaxation share the same optimal value of $\sqrt{3}-7/4\approx -0.018$, together with a simple analytic SDP optimal solution. More importantly, for the instance in this family with $a=b=1/2$, we further obtain a certificate point $(x^{\rm BQP}, X^{\rm BQP})\in\mathcal{S}_4$ with SDP value $-{1}/{100}$, whose off-diagonal projection $(x^{\rm BQP}, (X^{\rm BQP})_{i<j})$ belongs to the 4-dimensional Boolean quadric polytope (BQP), see \cite{padberg1989boolean,de1990cut}. Due to \cite{BurerLetchford2009}, any linear cut without $X_{ii}$ diagonal terms 
\begin{equation}
 \ell(x,X)=c_0+b^\top x+\sum_{1\le i<j\le n}h_{ij}X_{ij},
 \label{eq:BQP-general}
\end{equation}
including triangle, clique, and cycle inequalities, must be valid on the whole BQP, and hence cannot cut off the certificate $(x^{\rm BQP}, X^{\rm BQP})$, indicating an $\Omega(1)$ common lower bound for the relaxation gap of all SDP relaxation \eqref{eq:sdp-strengthen} formulated with only BQP valid cuts \eqref{eq:BQP-general}. 

Next, we move beyond the commonly adopted BQP valid cuts \eqref{eq:BQP-general} and consider SDP relaxations formulated with the most general linear cuts that contains diagonal terms:
\begin{equation}
\label{eq:general-cuts}
\ell(x,X) = c_0 +\sum_{i=1}^n b_i x_i
+\sum_{i=1}^n d_i X_{ii}
+\sum_{1\leq i<j\leq n}h_{ij}X_{ij}.
\end{equation}
Then, in dimension $n = 4$, by properly perturbing the base instance family \eqref{eq:bad-family} and leveraging the multivariate R\'emez inequality \citep{brudnyui1973extremal}, we provide an $\Omega(1/m^2)$ lower bound on the relaxation gap for SDP relaxation with $m$ general form linear cuts \eqref{eq:general-cuts}. 

Finally, by properly padding and duplicating the instance family \eqref{eq:bad-family} and its perturbed variant, we extend the 4-dimensional gap result to general dimension $n\geq 4$. For SDP relaxation \eqref{eq:sdp-strengthen} formulated with only BQP valid cuts \eqref{eq:BQP-general}, regardless of the number of cuts, there is an instance $(Q,c)\in\mathcal{N}_\infty^n$ with an $\Omega(n)$ relaxation gap. For SDP relaxation \eqref{eq:sdp-strengthen} formulated with general linear cuts \eqref{eq:general-cuts}, there exists an instance $(Q,c)\in\mathcal{N}_\infty^n$ with an $\Omega(n/m^2)$ relaxation gap. Both scales linearly with respect to the problem dimension $n$.

\section{Preliminary and related literature}

\subsection{Preliminary concepts and results}
\label{sec:preliminaries}

\noindent\textbf{Continuous submodular functions.}\,\, For \(x,y\in\mathbb{R}^n\), let \(x\vee y\) and \(x\wedge y\) be the
componentwise maximum and minimum, respectively. Then a function \(f:\mathbb{R}^n\to\mathbb{R}\) is submodular if $$f(x)+f(y)\ge f(x\vee y)+f(x\wedge y)$$ for all $x,y\in\mathbb{R}^n$. If \(f\) is twice continuously differentiable, then this definition is equivalent to $\nabla^2_{ij} f(x)\leq0$ for all $i\neq j$. If the full Hessian $\nabla^2 f(x)\leq0$ is entrywise nonpositive, then we say this submodular function satisfies the diminishing return (DR) property and is called DR-submodular. In the quadratic problem~\eqref{eq:boxqp}, its objective is submodular if and only if (iff.) $Q_{ij} \leq 0, ~ \forall i \neq j$, and is DR-submodular iff. $Q \leq 0$ componentwise. For more in-depth discussion of these concepts, please refer to \cite{bach2019submodular,bian2017continuous}, and see \cite{bilmes2022submodularity} for a more recent survey on the applications of submodularity in modern machine learning and artificial intelligence. \vspace{0.2cm}

\noindent\textbf{Valid linear cuts.}\,\, For all the linear constraints that we consider in the SDP relaxation \eqref{eq:sdp-strengthen} in the forms of either \eqref{eq:BQP-general} or \eqref{eq:general-cuts}, we assume they are \emph{valid linear cuts} in the following sense. 
\begin{definition}
\label{definition:Valid-Cuts}
We say a cut $\ell(x,X) \geq 0$ of form \eqref{eq:BQP-general} or \eqref{eq:general-cuts} is a \emph{valid linear cut} in the lifted variable $(x,X)$, if this inequality is valid for all \emph{rank-one lifted} points $(x,X)\in\{(x,xx^\top):x\in[0,1]^n\}$. 
\end{definition} 
Note that many widely used valid linear cuts can be derived from the Boolean quadric polytope
$$\mathcal{B}_n :=\mathrm{conv}\left\{ \bigl(y,(y_iy_j)_{i<j}\bigr):y\in\{0,1\}^n \right\}.$$
By \citet[Proposition~5 and Corollary~1]{BurerLetchford2009}, the off-diagonal projection of \(\mathcal M_n\) onto $(x,(X_{ij})_{i<j})$ is exactly \(\mathcal B_n\), and hence every valid linear cut on $\mathcal{M}_n$ with no diagonal terms $X_{ii}$, in the form of \eqref{eq:BQP-general}, must be valid on $\mathcal{B}_n$, and hence we call them the \emph{BQP valid cuts}.   \vspace{0.2cm}

\noindent\textbf{BQP valid cuts.}\,\, Though a general valid linear cut may contain diagonal terms $X_{ii}$ as in \eqref{eq:general-cuts}, most of the most commonly used cuts are BQP valid cuts \eqref{eq:BQP-general}. A typical example is the McCormick cuts. 
\begin{example}[McCormick inequalities]  Suppose that $(x,X)\in\{(x,xx^\top):x\in[0,1]^n\}$ is a rank-one lifted point, then, for $\forall i\neq j$, it satisfies the  McCormick inequalities:
\begin{equation}  
\label{eq:McCormick}
    X_{ij} \ge0, \quad  
    X_{ij} \ge x_i+x_j-1,\quad X_{ij} \le x_i, \quad 
    X_{ij} \le x_j.  
\end{equation}
\end{example}
We note that the linear cuts used in the SDP relaxations from \cite{bach2018efficient} and \cite{burer2025semidefinite} consist only of the two upper inequalities in \eqref{eq:McCormick}. As the McCormick cuts do not contain diagonal terms, they are obviously BQP valid. Other commonly used BQP-valid cuts include triangle, cycle, and clique inequalities \citep{barahona1986cut,padberg1989boolean,de1990cut,deza1997geometry}. Due to the page limit, we omit their detailed discussion.

\subsection{Related literature} \label{sec:related-literature}


This subsection reviews related results on SDP tightness, valid inequalities, and extended formulations.

For submodular BCQP with $c \leq 0$, the results of \citet{zhang2000quadratic,kim2003exact} imply that the Shor constraint together with proper linear constraints gives an SDP relaxation with the original optimal value. For general BoxQP, \citet{anstreicher2010computable} show that the SDP-RLT relaxation, which combines the Shor constraint with all box-RLT inequalities (including $X_{ii} \leq x_i$), exactly represents $\mathcal{M}_2$, but cannot represent $\mathcal{M}_n$ exactly with $n \geq 3$. \citet{qiu2024exact} then characterizes instances for which RLT or SDP-RLT achieves the original optimal value.  

For submodular BCQP with unrestricted $c$, as mentioned above, \citet{bach2018efficient} provides an $n = 3$ counterexample for a relaxation using pairwise $3 \times 3$ positive semidefinite constraints. Very recently, \citet{burer2025semidefinite} prove that
the SDP relaxation with $X \leq \bm{1} x^{\top}$ attains the original optimal value when $n \leq 3$. They also give an $n=4$ counterexample that maintains a positive relaxation gap after adding all RLT and triangle inequalities. 

Valid inequalities and extended formulations have also been used to strengthen convex relaxations of nonconvex QPs. 
\citet{vries2022tight} give a compact linear extended formulation for all A-odd-cycle inequalities on the graph defined by the nonzero off-diagonal entries of $Q$. More recently, \citet{anstreicher2025extended} propose extended triangle inequalities involving diagonal entries of $X$, together with second-order-cone (SOC) constraints in an extended space. \citet{dey2026second, khajavirad2026tight} give exact SOC and SDP reformulations of BoxQP, respectively, under conditions on the sparsity pattern and the locations of positive diagonal entries of $Q$. 

Our results quantify the limitations of adding valid linear cuts (Definition~\ref{definition:Valid-Cuts}) to the relaxation \eqref{eq:sdp-strengthen}. For $n \geq 4$ and any fixed family of $m$ valid linear cuts in $(x, X)$, there exists an instance $(Q, c) \in \mathcal{N}_\infty^n$ with gap $\Omega(n/(m + 1)^2)$ (Theorem~\ref{thm:Omega(n/m^2)}). When only BQP-valid cuts are imposed, an $\Omega(n)$ SDP relaxation gap remains even if all such cuts are included (Theorem~\ref{thm:Omega(n)-BQP-GP}). In the next section, we start with a four-dimensional construction for BQP-valid cuts.

\section{An $\Omega(1)$ relaxation gap for SDP with BQP valid cuts in dimension $n=4$} 
\label{sec:BQP-nontightness}
\subsection{The basic submodular BCQP construction}
We start by discussing the theoretical properties of the hard instance family \eqref{eq:bad-family}. The numerical methods for constructing the instances, together with the convex set inclusion condition for SDP tightness that induced these methods, are all relegated to Section \ref{sec:n4-counterexample}.

\begin{proposition}
\label{prop:QP-opt=0}
    For any $a,b\in\left[\frac{1}{2},1\right]$, the BCQP instance \eqref{eq:boxqp} with $Q,c$ selected according to \eqref{eq:bad-family} is submodular and has optimal value $0$. 
\end{proposition} 
\begin{proof}
First of all, depending on the value of $x_3$, the QP objective function allows a natural decomposition into nonnegative sums. Denote $f(x)=x^\top Qx+c^\top x$. In case $x_3\in[0,1/2]$, we have
$$f(x)=(ax_1-x_2)^2+\left(x_3-\frac{x_4}{2}\right)^2+ax_1+x_2(1-2x_3)+(2b-1)x_3(1-x_4)\geq0$$
because each term is nonnegative. In case $x_3\in(1/2,1]$, we can rewrite the quadratic objective function in the following nonnegative sums: 
$$
f(x)=\left(ax_1-x_2+x_3-\frac12\right)^2 + 2ax_1(1-x_3)  +\frac{(1-x_4)(4x_3-x_4-1)}{4}  +(2b-1)x_3(1-x_4)\geq0.$$
Hence $f(x)\geq0, \forall x\in[0,1]^4$. On the other hand, we have $f(0)=0$, which completes the proof. 
\end{proof} 

We should note that although $x=0$ is a global optimal solution for this family of instances, it may not be the unique solution. For example, take $b=1/2$, the $f(x)\equiv0$ on the line segment $x^*(t) = (0,0,t,2t)$ with $t\in[0,1/2]$. 

\subsection{An exact gap for the McCormick baseline}
In the next proposition, let us show that the SDP relaxation \eqref{eq:sdp-strengthen} with McCormick cuts \eqref{eq:McCormick} admits an exact relaxation gap of $\sqrt{3}-7/4 < 0$.   

\begin{proposition}
\label{prop:SDP-opt<0}
For any $a,b\in\left[\frac{1}{2},1\right]$, the SDP relaxation \ref{eq:sdp-strengthen} with McCormick cuts allows an exact optimal solution $(x^*(a),X^*(a))$, independent of $b$, given by    
\[
x^*(a)=
\begin{pmatrix}
\frac{\gamma}{3a}\\[1mm]
\frac{1}{3}\\[1mm]
\frac{2}{3}\\[1mm]
1-\frac{2\gamma}{3}
\end{pmatrix},\,\,
X^*(a)=
\begin{pmatrix}
\frac{\gamma}{6a^2}&\frac{\gamma}{3a}&\frac{\gamma}{3a}&
\frac{1-3\gamma}{3a}\\[1mm]
\frac{\gamma}{3a}&\frac{1-\gamma}{3}&\frac{1}{3}&\frac{1}{3}\\[1mm]
\frac{\gamma}{3a}&\frac{1}{3}&\frac{2-\gamma}{3}&\frac{2}{3}\\[1mm]
\frac{1-3\gamma}{3a}&\frac{1}{3}&\frac{2}{3}&1-\frac{2\gamma}{3}
\end{pmatrix}
\]
where $\gamma := 2-\sqrt{3}$ is used to shorten the formula. For $\forall a,b\in\left[\frac{1}{2},1\right]$, the optimal value of the SDP relaxation is $\sqrt{3}-\frac{7}{4}<0$.
\end{proposition}
\begin{proof} 
Denote the lifted matrix of $x^*(a)$ and $X^*(a)$ as
$$Y^*(a):=\left(\begin{aligned}
    \,1\quad & \,x^*(a)^\top\\[-1mm]
    x^*(a)\,\, & X^*(a)
\end{aligned}\right).$$ Then it is straightforward to verify that $Y^*(\frac{1}{2})\succeq0$. Let $D(a):=\mathrm{Diag}(1,1/2a,1,1,1)$, then 
$$Y^*(a) = D(a)Y^*\left(\frac{1}{2}\right)D(a)\succeq0$$
and thus the positive semidefinite constraint \eqref{eq:sdp-strengthen} is satisfied. Then by direct computation, we can verify that for all $a\in[1/2,1]$, this point satisfies all the McCormick inequalities, and is hence a feasible solution. Finally, substituting the $x^*(a),X^*(a)$ to SDP objective function and we will observe that both $a$ and $b$ are canceled out: 
\begin{equation} 
Q(a,b)\bullet X^*(a) + c(a,b)^\top x^*(a)\equiv\sqrt{3}-\frac{7}{4}.\nonumber
\end{equation} 
It remains to lower the optimal value by $\sqrt{3}-\frac{7}{4}$ through weak duality. Let us fix any $a,b\in\left[\frac{1}{2},1\right]$. For $(x,X)\in\mathbb{R}^n\times\mathbb{S}^n$, denote $Y$ its lifted point, and define the Lagrangian function
$\mathcal{L}(x,X,\lambda,\Lambda)$ by
$$\mathcal{L}(x,X,\lambda,\Lambda)\,=\, Q(a,b)\bullet X + c(a,b)^\top x  - Y\bullet\Lambda - \sum_i \lambda_i\ell_i(x,X)$$
where $\ell_i$ corresponds to the constraints in McCormick inequalities \eqref{eq:McCormick}. Then, by weak duality:
$$V^*_{\rm Mc}
\geq  \min_{{x\in\mathbb{R}^n},{X\in\mathbb{S}^n}} \mathcal{L}(x,X,\lambda,\Lambda),\,\,\,\forall {\lambda\geq0},{\Lambda\succeq0},$$
where $V^*_{\rm Mc}$ stands for the optimal value of the McCormick SDP relaxation. Now construct such a pair of feasible dual variables. Keep $\gamma = 2-\sqrt{3}$, we set 
\[
A=\begin{pmatrix}
-\gamma & 1 & \gamma-2 & 1 & 0 \\[1mm]
0 & 0 & 1 & \gamma-2 & 1
\end{pmatrix},
\]
and 
$$\Lambda^*(a) = \frac{1}{4}\left[D(a)\right]^{-1}A^\top A\left[D(a)\right]^{-1}\succeq0$$ is positive semidefinite. Then set $$\lambda^*_{12}=a(2-\sqrt{3}),\quad  \lambda^*_{13}=a,\quad  
\lambda^*_{23}=2-\sqrt{3},$$
$$\lambda^*_{24}=\frac{1}{2}, \qquad \lambda^*_{34}=2b-\frac{\sqrt{3}}{2},$$
and set all the other $\lambda$-multipliers to 0, then $\lambda\geq0$. Through direct computation, all the terms involving $a, b$ and $x, X$ can be canceled out, and the following identity always holds as long as $X\in\mathbb{S}^n$:
\begin{align*}
\mathcal{L}(x,X&, \lambda^*,\Lambda^*) =  Q(a,b)\bullet X + c(a,b)^\top x - Y \bullet \Lambda^*(a)- \lambda^*_{12}(x_1-X_{12}) \\
& -\lambda^*_{13}(x_1-X_{13})-\lambda^*_{23}(x_2-X_{23}) -\lambda^*_{24}(x_2-X_{24}) - \lambda^*_{34}(x_3-X_{34}) \equiv \sqrt{3}-\frac{7}{4}.
\end{align*} 
Together with the above weak duality argument, we complete the proof.  
\end{proof}
\begin{corollary}
    By incorporating the McCormick inequalities, it is not possible to close the relaxation gap between the submodular BCQP \eqref{eq:boxqp} and the SDP relaxation \eqref{eq:sdp-strengthen} in dimension $n=4$. 
\end{corollary}

\subsection{The $\Omega(1)$ gap for all SDP with sheer BQP valid cuts}
In Proposition \ref{prop:SDP-opt<0}, with only McCormick inequalities, the SDP relaxation allows a structured analytic form optimal solution. However, such a solution can potentially be cut off by additional cuts. For example, when $a=b=1/2$, the solution $x^*(a),X^*(a)$ can be cut off by the triangle inequality. Therefore, we would like to systematically assess the approximation quality of all BQP valid cuts, which is sufficient to consider the following SDP in dimension $n=4$:
\begin{align}
    &\mathop{\rm minimize}_{(x,X)\in\mathcal{S}_4, y\in\Delta_{16}} \,\,\,  Q\bullet X + c^\top x  \nonumber\\
    &\qquad\mathrm{s.t.} \qquad\,\,\, (x,(X_{ij})_{i<j}) = y_1P_1+\cdots y_{16}P_{16}, \label{eq:SDP-BQP-certificate} 
\end{align}
where $\Delta_{16}:=\left\{y\in\mathbb{R}^{16}: \sum_iy_i=1, y\geq0\right\}$ denotes the probability simplex, and $P_1,\cdots,P_{16}$ denote the sixteen vertices of the Boolean quadric polytope $\mathcal{B}_4$. Select $Q$ and $c$ from the instance family \eqref{eq:bad-family} with $a=b=1/2$, then a direct computation gives an optimal value $v^*\approx -0.0108<0$, whereas the optimal value of the QP problem equals $0$, see Proposition \ref{prop:QP-opt=0}. As the optimal solution to \eqref{eq:SDP-BQP-certificate} belongs to $\mathcal{B}_4$, it is automatically feasible to all other BQP valid cuts of form \eqref{eq:BQP-general}. Nevertheless, for mathematical rigor, the next proposition provides an exact rational certificate which is exactly feasible to \eqref{eq:SDP-BQP-certificate} and has strictly negative objective value.

\begin{proposition}
\label{prop:BQP-certificate}
Let $(Q, c)$ be selected from the parameterized family \eqref{eq:bad-family} with $a=b=1/2$. Then \eqref{eq:SDP-BQP-certificate} has a feasible solution $x^{\rm BQP}, X^{\rm BQP}, y^{\rm BQP}$ with objective value $-1/100$.
\end{proposition}
The values of $x^{\rm BQP}, X^{\rm BQP}, y^{\rm BQP}$ and their feasibility verification are attached in the Appendix~\ref{app:BQP-certificate}. In fact, this certificate also satisfies the general form constraint $X_{ii}\leq x_i, i\in[n].$ And hence is not affected by these easily obtainable cuts of form \eqref{eq:general-cuts}.

\begin{theorem}
\label{thm:BQP-n4} 
In dimension $n=4$, if an SDP relaxation \eqref{eq:sdp-strengthen} is formulated only with cuts of form \eqref{eq:BQP-general}, then there exists an instance $(Q,c)\in\mathcal{N}_\infty^4$ with relaxation gap at least $1/100$.
\end{theorem}

\section{An $\Omega(1/m^2)$ relaxation gap for SDP with $m$ general cuts in dimension $n=4$}
\label{sec:finite-cuts}

In Section \ref{sec:BQP-nontightness}, we have constructed a family of instances $\big(Q(a,b), c(a,b)\big)$ such that an exact relaxation gap of $\sqrt{3}-7/4\approx-0.018$ exists for SDP relaxation with  McCormick inequalities. Then in Proposition \ref{prop:BQP-certificate}, we have constructed an exact certificate point in $\mathcal{B}_4$ with a gap of $-1/100$, disproving the tightness of any set of BQP valid cuts. However, for any pair of $({Q},{c})$ from the parametrized family \eqref{eq:bad-family} with $a,b\in[1/2,1]$, adding a single (optimal) objective value cut 
$${Q}\bullet X + {c}^\top x\geq  0$$
to \eqref{eq:sdp-strengthen} is sufficient to close the relaxation gap for this family of hard instances. Therefore, it remains to study whether more general cuts of form \eqref{eq:general-cuts} can retain the tightness of SDP. Consider a slightly perturbed variant of the parameterized family \eqref{eq:bad-family}, in which a strongly convex term is added to introduce solution uniqueness to the QP problem. For the ease of presentation, we restrict ourselves to the special choice $a=b=1/2$, and construct the perturbed instance
\begin{equation}
\label{eq:bad-family-perturbed}
Q^\eta:=Q\left(\frac{1}{2},\frac{1}{2}\right) + \eta\cdot E_{44},\quad c^\eta = c\left(\frac{1}{2},\frac{1}{2}\right)-\eta\cdot e_4,
\end{equation} 
where $E_{44}$ stands for the unit matrix in $\mathbb{R}^{4\times 4}$ with only the (4,4)-th entry equal to 1 and all other entries equal to 0, and $e_4$ stands the unit vector with only the 4-th entry equal to 1 and all other entries equal to 0. This correspond to adding the term $\eta(x_4-1/2)^2-\eta/4$ to the objective function of the unperturbed QP instance. For this instance, the following proposition holds. 

\begin{proposition}
\label{prop:QP-opt=-eta/4}
For any $\eta>0$, the BCQP instance \eqref{eq:boxqp} with $Q,c$ selected according to \eqref{eq:bad-family-perturbed} is submodular. It has optimal value $-\eta/4$, and the optimal solution $x^*=(0,0,1/4,1/2)$ is unique. 
\end{proposition} 
\begin{proof}
Denote $f_\eta(x)=x^\top Q^\eta x+x^\top c^\eta$ the objective function of BCQP. Then we have $$f_\eta(x) = f_0(x) + \eta(x_4-1/2)^2-\eta/4$$
where $f_0(x)$ corresponds to the original QP instance discussed in Proposition \ref{prop:QP-opt=0}, with no perturbation. Consequently, we have $f_0(x)\geq0$ for $\forall x\in[0,1]^4$, and hence $f_\eta(x)\geq-{\eta}/4$ for $\forall x\in[0,1]^4$. On the other hand, $f_\eta(x^*) = -\eta/4$, proving the optimality of the point $x^*$ stated in the proposition. The key is to show the uniqueness of the solution. As the $f_0$ part is nonnegative, to achieve the minimum, we must have $x_4^* = 1/2$. Then in case $x_3\in(1/2,1]$ and with $x_4^* = 1/2$, we have 
$$
f_\eta(x)=\left(\frac{x_1}{2}-x_2+x_3-\frac12\right)^2 + x_1(1-x_3)  + \frac{4x_3-3/2}{8} - \frac{\eta}{4}>\frac{1}{16}-\frac{\eta}{4}.$$
Hence the optimality can only be achieved in case $x_3\in[0,1/2]$. In this case, with $x_4^* = 1/2$, we have
$$f_\eta(x)=\left(\frac{x_1}{2}-x_2\right)^2+\left(x_3-\frac{1}{4}\right)^2+\frac{x_1}{2}+x_2(1-2x_3)-\frac{\eta}{4}.$$
Because every term is nonnegative except for the constant $-\eta/4$. They must all be 0 at optimal solution, and this leads to the uniqueness of the solution $x^*=(0,0,1/4,1/2)$. 
\end{proof}

Denote $L_\eta(x,X):=Q^\eta\bullet X + x^\top c^\eta$. 
Because the matrix $Q^\eta$ is not positive semidefinite because its leading $3\times3$ principal submatrix has determinant $-1/4$. One can always find some $(\bar{x},\bar{X})\in\mathcal{S}_4$ with objective value $L_\eta(\bar{x},\bar{X}) < f_\eta(x^*) = -\eta/4$. Therefore, the following basic observation reveals that any cut strictly feasible at $(x^*,x^*(x^*)^\top)$ cannot close the relaxation gap.

\begin{proposition}
\label{prop:strict-case}
    Let $\ell_1, \cdots, \ell_m$ be a finite family of valid linear cuts and let $x^*$ be specified by Proposition \ref{prop:QP-opt=-eta/4}. Suppose the strict feasibility $\ell_i(x^*,x^*(x^*)^\top)> 0$ holds for all $i=1,\cdots,m$. Then fix any $(\bar{x},\bar{X}) \in \mathcal{S}_4$ s.t. $L_\eta(\bar{x},\bar{X})<f_\eta(x^*)$, there exists some $\epsilon \in(0,1]$ s.t. the point 
    $$(x^\epsilon,X^\epsilon) := \epsilon \cdot (\bar{x},\bar{X})+(1-\epsilon) \cdot (x^*,x^*(x^*)^\top)$$
    satisfies $L_\eta(x^\epsilon,X^\epsilon) < f_\eta(x^*)$ and 
     $\ell_i(x^\epsilon,X^\epsilon)>0$ for all $i=1,\cdots,m$.  
\end{proposition}

The proof of Proposition~\ref{prop:strict-case} is straightforward and is hence omitted. It shows that only the cuts that are active at $(x^*,x^*(x^*)^\top)$ may have the opportunity to close the relaxation gap.  
\begin{remark}
\label{remark:perturbation}
Let $F$ be the minimal box face containing \(x^*\) with two free coordinates $x_3$ and $x_4$: 
\begin{equation}
F :=\left\{x\in[0,1]^4: x_1=x_2=0\right\} ~ .
\label{eq:F0}
\end{equation}
Then Proposition~\ref{prop:strict-case} provides a clear intuition on why we construct the perturbed instance \eqref{eq:bad-family-perturbed} so that $x^*$ stays in the relative interior of $F$. In this case, the coordinates $x^*_3$ and $x_4^*$ are sensitive to perturbation in $Q$ and $c$, while the uniqueness enables precise control of the response. Therefore, we may properly design a perturbed instance $(\tilde{Q},\tilde{c})$ around $(Q^\eta,c^\eta)$ so that for the new QP optimal solution $\tilde{x}^*$, the cuts that are once active at $(x^*,x^*(x^*)^\top)$ but \emph{not the whole face} $F$ become strictly feasible at $(\tilde{x}^*,\tilde{x}^*(\tilde{x}^*)^\top)$, giving a strictly positive relaxation gap on the perturbed instance $(\tilde{Q},\tilde{c})$.
\end{remark}

Though how to construct this appropriately perturbed instance remains nontrivial, it is not a conceptual challenge and shall be presented later. Therefore, the prospective cuts that may potentially close the relaxation gap can then be further limited to the so-called \emph{face-active} cuts, defined as follows. 
\begin{definition}
\label{definition:Face-active-cuts}
We say a valid linear cut $\ell(x, X) \geq 0$ is \emph{face-active} with respect to any given face $F$, if this cut is active on the entire face $F$. 
\end{definition}
Interestingly, these face-active cuts exhibit  an interesting geometry that directly points out the construction method of a numerical certificate.  In detail, consider a general linear cut of form \eqref{eq:general-cuts}. Set $q(x):=\ell(x,xx^\top)$ a quadratic function of $x$. Then let $\ell$ be an arbitrary valid linear cut that is face-active on the entire $F$, we have 
$$q(0,0,x_3,x_4) = c+b_3x_3+b_4x_4 +d_3x_3^2+h_{34}x_3x_4+d_4x_4^2\equiv0$$
for $\forall (x_3,x_4)\in[0,1]^2$. Note that the following fact holds, and will repeated used in later construction.  
\begin{remark} \label{remark:quadratic-zero-set}
Let $g:\mathbb{R}^2\to\mathbb{R}$ be a \emph{nonzero} bivariate quadratic function. Then, geometrically, the zero set $\{z\in\mathbb{R}^2: g(z)=0\}$ can only be one of the following: empty set, singleton, one line, a pair of two lines, hyperbola, parabola, or ellipse, among which none of them has an interior on $\mathbb{R}^2$. 
\end{remark}
As a result, the only possibility for $q(0,0,x_3,x_4)=0$ to hold on the entire face $F$ is that the whole function is 0, that is, we must have $$c=b_3=b_4=d_3=h_{34}=d_4=0,$$
while only the remaining nine coefficients  
$$\alpha = \left(b_1,b_2,h_{13},h_{14},h_{23},h_{24},d_1,d_2,h_{12}\right)^\top$$ could be nonzero. Define
\begin{equation}
\label{eq:psi-Psi}
\begin{aligned}
    &\psi(x):= \big(x_1, x_2, x_1x_3, x_1x_4, x_2x_3, x_2x_4, x_1^2, x_2^2, x_1x_2\big)^{\top}\\
    &\Psi (x,X):= \big(x_1, x_2, X_{13}, X_{14}, X_{23}, X_{24}, X_{11}, X_{22}, X_{12}\big)^{\top}
\end{aligned}
\end{equation}
where $\Psi$ corresponds to the lifted form of $\psi$ so that $\Psi(x,xx^\top)=\psi(x)$. The above discussion suggests that any valid linear cut $\ell$ that is face-active on the entire $F$ can be written as 
\begin{equation}
    \label{eq:Psi}
    \ell(x,X) = \alpha^\top\Psi(x,X),
\end{equation}
for some $\alpha\in\mathbb{R}^9$. Thus the quadratic function $q(x) = \ell(x,xx^\top)$ induced by the linear cut $\ell$ can also be written as
$q(x) = \alpha^\top\psi(x)$ with the same coefficient $\alpha.$ Finally, notice that we require $\ell$ to be a \emph{valid} linear cut such that $q(x)\geq0$ for all $x\in[0,1]^4$. This immediately implies that the coefficient $\alpha$ must stay in the dual cone of the following cone in $\mathbb{R}^9$: 
\begin{equation}
\label{eq:cone}
\mathcal{K} := \operatorname{cone} \left\{ \psi(x):x\in[0,1]^4 \right\}.
\end{equation}
We summarize the above observation as the next lemma. 
\begin{lemma}
\label{lemma:dual-cone-alpha}
    Let $\ell$ be an arbitrary valid linear cut in the form \eqref{eq:general-cuts}. If this cut is face-active on the entire $F$ given in \eqref{eq:F0}, then its coefficients $c,b_3,b_4,d_3,h_{34},d_4$ must all vanish, and $\ell$ must take the form \eqref{eq:Psi} with some $\alpha\in\mathcal{K}^*$, with the dual cone defined by $\mathcal{K}^*:=\left\{y:y^\top z\geq0,\ \forall z\in\mathcal{K}\right\}.$
\end{lemma}

Therefore, whether there exists a set of face-active cuts that can remove all the candidate points $(\bar{x},\bar{X})$ in Proposition \ref{prop:strict-case} reduces to the \emph{strict} feasibility problem for the following set: 
\begin{align*}
\left\{(x,X)\in\mathcal{S}_4\,: \,\Psi(x,X)\in \mathcal{K},\, L_\eta(x,X) <-\frac{\eta}{4}\right\}.
\end{align*} 
Therefore, if this set is nonempty, then no collection of face-active cuts can remove all candidate points, otherwise the problem still remains open. This type of problem is in general not easy to solve. Yet fortunately, for this specific problem (under relative low dimension), it turns out that we can find such a strictly feasible solution by solving the following SDP inner approximation:
\begin{equation}
\label{eq:conic-certificate-SDP}
\begin{aligned}
    & \mathop{\rm minimize}_{(x,X)\in\mathcal{S}_4,\lambda\in\mathbb{R}_+^{64}} \,\,\, Q^\eta\bullet X \,+\, x^\top c^\eta\\
    & \qquad \mathrm{s.t.} \quad \Psi(x,X) = \lambda_1 \psi(\xi^1) +\cdots  + \lambda_{64}\psi(\xi^{64}),
\end{aligned}
\end{equation}
where the 64 points $\xi^s\in[0,1]^4$ are constructed by selecting the first two coordinates $\xi_{1},\xi_{2}\in\left\{0,\frac13,\frac23,1\right\}$ and the last two $\xi_{3},\xi_4\in\{0,1\}$.  The next proposition provides such a certificate. For the simplicity of the certificate, rather than the optimal solution to \eqref{eq:conic-certificate-SDP}, we select a clean feasible solution  with sufficiently small objective value.\footnote{When numerically solving \eqref{eq:conic-certificate-SDP} for a certificate, one can fix any small enough $\eta$, e.g., $\eta=1/100$, and then find a clean rational certificate around it. Finer grids, e.g. $\{0,1/4,1/2,3/4,1\}$, may produce a certificate with larger gap, but is unnecessary in this proof.}

\begin{proposition}
\label{prop:conic-certificate}
There is an exact feasible solution $(\bar{x},\bar{X},\bar{\lambda})$  to problem \eqref{eq:conic-certificate-SDP} such that $$L_\eta(\bar{x},\bar{X}) - f_\eta(x^*)< -1/200 $$
as long as we select the perturbation parameter $\eta \leq 1/100$. 
\end{proposition} 

The detailed values the certificate $(\bar{x},\bar{X},\bar{\lambda})$ are provided in Appendix \ref{app:conic-certificate}. The verification only contains simple numerical computation and is hence omitted. This proposition suggests that there exists a certificate $(\bar{x},\bar{X})$ that is, unfortunately, impossible to be cut off by the face-active constraints.

Combining the observations in Proposition \ref{prop:strict-case}, Remark \ref{remark:perturbation}, and Proposition \ref{prop:conic-certificate}, we are able to claim that there is no finite set of valid linear cuts can close the SDP relaxation gap for the submodular BCQP \eqref{eq:boxqp}, in dimension $n=4$. It now remains to provide the precise construction of the perturbed instance promised in Remark \ref{remark:perturbation}. 

\begin{lemma}
    \label{lemma:perturb-optimality}
    Let us fix any $\eta\leq1/100$. Let $x^*(\theta):=(0,0,\theta_1,\theta_2)$ for some parameter $\theta\in\mathbb{R}^2$ and let $\theta^*:=(1/4,1/2)=(x^*_3,x^*_4)$ s.t. $x^*(\theta^*) = x^*$. Then there exists an open neighborhood $\Theta^\eta\ni\theta^*$ s.t. for $\forall\theta\in\Theta^\eta$, $x^*(\theta)$ is the global optimal solution of the perturbed instance $(Q^{\eta,\theta},c^{\eta,\theta})$ with  
    $$Q^{\eta,\theta}=Q^\eta,\quad c_{I}^{\eta,\theta} = c_{I}^{\eta},\quad c_{J}^{\eta,\theta} = c_{J}^{\eta} - 2Q_{JJ}^\eta(\theta-\theta^*),$$ 
    where $I=\{1,2\}$ and $J=\{3,4\}$. By choosing $\Theta^\eta$ small enough, we can guarantee for all $\theta\in\Theta^\eta$ that 
    $$L_{\eta,\theta}(\bar{x},\bar{X})-f_{\eta,\theta}(x^*(\theta))<-1/250,$$
    where $(\bar{x},\bar{X})$ is the certificate in Proposition \ref{prop:conic-certificate}, $L_{\eta,\theta}(x,X):=Q^{\eta,\theta}\bullet X+x^\top c^{\eta,\theta}$ and $f_{\eta,\theta}(x):=x^\top Q^{\eta,\theta}x + x^\top c^{\eta,\theta}$ denotes the SDP and QP objective functions of the perturbed instance, respectively.  
\end{lemma}

Intuitively, note that $Q_{JJ}^{\eta,\theta} = Q_{JJ}^{\eta}\succ0$ always hold, and the perturbed vector $c^{\eta,\theta}$ is constructed by solving the KKT condition $\nabla_{J} f_{\eta,\theta}(x^*(\theta))=0$. Because $\nabla_1f_{\eta}(x^*) = \nabla_2f_{\eta}(x^*) = 1/2>0$, slightly perturbing $\theta$ will maintain their positiveness, and will also not change the optimal active face. Then the lemma can be proved. However, instead of such a general argument, our quadratic objective function allows a simpler proof, as provided in Appendix \ref{app:perturb-optimality}. 

\begin{theorem}
\label{thm:Gap-General}
In dimension $n = 4$, there exists an absolute constant $\delta_{\rm Gap}>0$ such that for an arbitrary set of fixed valid linear cuts $\{\ell_1, \cdots, \ell_m\}$, there exists a submodular BCQP instance $(Q,c)$ for which the SDP relaxation \eqref{eq:sdp-strengthen} has a strictly positive relaxation gap of at least ${\delta_{\rm Gap}}/{(m+1)^2}$. 
\end{theorem}

\begin{proof} Fix $\eta=1/100$ for simplicity, we partition the set of valid linear cuts into two subsets. Define $J_1:= \left\{j\in[m]: \ell_j(x, x(x)^\top)= 0, \forall x\in F\right\}$ as the set of face-active cuts on $F$, and let $J_2 := [m]\backslash J_1$. 
Then if $J_2 = \emptyset$, then we can directly work with the unperturbed instance $(Q^\eta,c^\eta)$. Otherwise, we switch to a perturbed instance $(Q^{\eta,\theta},c^{\eta,\theta})$ for some appropriately selected $\theta\in\Theta^\eta$.

Now suppose that $J_2\neq\emptyset$. For a linear cut $\ell = \ell_j$ with $j\in J_2$, $q(\theta):=\ell(x^*(\theta),x^*(\theta)x^*(\theta)^\top)$ will be a nonzero bivariate quadratic function of the form: 
\begin{equation}
    \label{eq:q-theta}
    q_\ell(\theta) = c+b_3\theta_1+b_4\theta_2 +d_3\theta_1^2+h_{34}\theta_1\theta_2+d_4\theta_2^2,
\end{equation}
where the coefficients are written in consistence with \eqref{eq:general-cuts}. Note that $\Theta^\eta$ is an open set, we can select a compact rectangular $D\subseteq\Theta^\eta$ with nonempty interior. Then, there exists $C_D>0$ such that
\begin{equation}
    \label{eq:Certificate-lower-bound}
    \ell(\bar{x},\bar{X})\geq - C_D\cdot\|q_\ell\|_{D},\,\,  \forall \ell
\end{equation}
where $(\bar{x},\bar{X})$ is given in Proposition \ref{prop:conic-certificate}, and $\|q\|_{D}:=\max\left\{|q(\theta)|:\theta\in D\right\}$ can be viewed as a norm over the six-dimensional linear space of bivariate quadratic functions. The proof relies on the detailed numerical values of the certificate, and is hence provided in Appendix \ref{app:certificate-lower-bound}. On the other hand, for a valid linear cut with $q(\theta)\geq0, \forall \theta\in D$, the multivariate R\'emez inequality \cite{brudnyui1973extremal} indicates that the area in which the ``feasibility margin'' is small cannot be very large for non-face-active cuts:
\begin{equation}
\label{eq:Remez}
\frac{\big|\!\big\{\theta\!\in\! D\!:\! q_\ell(\theta)\leq \varrho\|q_\ell\|_{D}\big\}\!\big|}{\sqrt{\varrho}\cdot |D|} \leq C_R,  \forall  q_{\!\ell}\!\neq0, \varrho\!\in\!(0,\!1]
\end{equation}
where $|\cdot|$ denotes the Lebesgue measure in $\mathbb{R}^2$, and $C_R$ is a constant for Remez inequality that relies only on the degree of $q_\ell$ and the dimension of $\theta$ (both are 2). Then set $\varrho = \frac{1}{4C_R^2m^2}$ and apply union bound gives $$\bigg|\bigcup_{j\in J_2}\{\theta\in D: q_{\ell_j}(\theta)\leq \varrho\|q_{\ell_j}\|_{D}\}\bigg| \leq |D|/2.$$ 
We can then select a $\theta\in D$ such that
\[\ell_j(x^*(\theta), x^*(\theta)x^*(\theta)^\top) = 0, \quad\qquad\,\,\, \mbox{for}\quad \forall j\in J_1,\]
\[\ell_j(x^*(\theta), x^*(\theta)x^*(\theta)^\top) \,> \frac{\|q_{\ell_j}\|_D}{4C_R^2m^2},\! \quad\mbox{for}\quad \forall j\in J_2.\] 
Consequently, take the  certificate $(\bar{x},\bar{X})$ in Proposition \ref{prop:conic-certificate}. Then $\ell_j(\bar{x},\bar{X})\geq0$ for $\forall j\in J_1$, and $f_{\eta,\theta}(x^*(\theta))-L_{\eta,\theta}(\bar{x},\bar{X})>1/250$ by Lemma \ref{lemma:perturb-optimality}. Then it remains to follow the interpolation approach in Proposition \ref{prop:strict-case} and choose $\epsilon = (4C_DC_R^2m^2+1)^{-1}$ so that the point 
$$(x^\epsilon,X^\epsilon) = \epsilon\cdot(\bar{x},\bar{X}) + (1-\epsilon)\cdot(x^*(\theta), x^*(\theta)x^*(\theta)^\top)$$
is feasible to all cuts, and $$f_{\eta,\theta}(x^*(\theta))-L_{\eta,\theta}(x^\epsilon,X^\epsilon) > \epsilon/250 \geq \delta_{\rm Gap}/(m+1)^2$$
for some $\delta_{\rm Gap}>0$. 
\end{proof}


\section{An $\Omega(n)$ Relaxation Gap for General Dimension $n\geq4$}
In this section, through the technique of zero padding and duplication, we further extend the current relaxation gap lower bounds to arbitrary dimension $n\geq4$, with an $\Omega(n)$ relaxation gap for BQP valid cuts, and an $\Omega(n/m^2)$ relaxation gap for $m$ general valid cuts. In details, for any hard instance $\big(Q^{(4)},c^{(4)}\big)\in \mathcal N_\infty^{4}$, we define its  $k$-fold duplication for some $k\geq1$ as  
\begin{equation}
\label{eq:QC-duplicate}
Q^{(4k)} = I_k\otimes Q^{(4)}\quad\mbox{and}\quad  c^{(4k)} = \mathbf{1}_k\otimes c^{(4)}, 
\end{equation}
where $I_k$ denotes the $k\times k$ identity matrix, $\mathbf{1}_k$ denotes the $k\times 1$ all-one vector, and $\otimes$ denote the Kronecker product. Denote $f_n(x) = x^\top Q^{(n)}x + x^\top c^{(n)}$ the QP objective function and $p^{(n)}$ the corresponding minimum value over $[0,1]^n$, then \eqref{eq:QC-duplicate} basically duplicates the $n=4$ instance for $k$ times: $f_{4k}(x) = \sum_{i=1}^k f_4(x^i)$ with each $x^i\in[0,1]^4$ being an independent decision variable. Therefore, we have $p^{(4k)} = k\cdot p^{(4)}$ always holds. Then, based on whether the set of $m$ linear cuts are BQP valid cuts or general valid linear cuts, it remains to construct the $k$-fold duplicated certificate points, respectively. For notational simplicity, we also denote $L_n(x,X)=Q^{(n)}\bullet X + x^\top c^{(n)}$.  \vspace{0.2cm}

\noindent\textbf{General duplication.}
Suppose we are encountering $m$ general valid linear cuts $\ell^{(4k)}_1,\cdots,\ell_m^{(4k)}$ in dimension $4k$. For any lifted variable $\big(x^{(4)}, X^{(4)}\big) \in \mathbb{R}^{4} \times \mathbb{S}^{4}$, we define its $k$-fold duplication as 
\begin{equation}
\label{eq:Gen-Duplicate}
x^{(4k)} = \mathbf{1}_k\otimes x^{(4)}\quad\mbox{and}\quad  X^{(4k)} = \mathbf{1}_{k\times k}\otimes X^{(4)}, 
\end{equation}
where $\mathbf{1}_{k\times k} = \mathbf{1}_k(\mathbf{1}_k)^\top$ denotes the $k\times k$ all-one matrix. Now, fix this set of cuts, we consider the set of cuts in dimension 4 induced by the duplication scheme \eqref{eq:Gen-Duplicate}: $\ell^{(4)}_i(x^{(4)}, X^{(4)}):=\ell^{(4k)}_i(x^{(4k)}, X^{(4k)})$. Then $\ell_1^{(4)},\cdots, \ell^{(4)}_m$ are a set of valid linear cuts in dimension 4. By Theorem \ref{thm:Gap-General} and its proof, there exists a certificate $\big(\bar{x}^{(4)}, \bar{X}^{(4)}\big)\in\mathcal{S}_4$ and an instance $\big(\bar Q^{(4)},\bar c^{(4)}\big)\in \mathcal N_\infty^{4}$ such that 
$$p^{(4)}-L_4\big(\bar{x}^{(4)}, \bar{X}^{(4)}\big)>{\delta_{\rm Gap}}/{(m+1)^2}.$$
As this certificate point is feasible to the SDP relaxation, its $k$-fold duplication in the form of \eqref{eq:Gen-Duplicate} will also be feasible to the original SDP relaxation in dimension $4k$ because 
$$\ell^{(4k)}_i\big(\bar{x}^{(4k)}, \bar{X}^{(4k)}\big)=\ell^{(4)}_i\big(\bar{x}^{(4)}, \bar{X}^{(4)}\big)\geq0, \forall i\in[m],$$
and the Schur complement $\bar{X}^{(4k)}-\bar{x}^{(4k)}(\bar{x}^{(4k)})^\top = \mathbf{1}_{k\times k}\otimes(\bar{X}^{(4)}-\bar{x}^{(4)}(\bar{x}^{(4)})^\top)\succeq0$. Hence, for the $k$-fold duplicated instance $(\bar Q^{(4k)},\bar c^{(4k)})$ constructed in the form \eqref{eq:QC-duplicate}, we have 
$$p^{(4k)}-L_{4k}\big(\bar{x}^{(4k)}, \bar{X}^{(4k)}\big) = k\cdot \left(p^{(4)}-L_4\big(\bar{x}^{(4)}, \bar{X}^{(4)}\big)\right) >{k\cdot\delta_{\rm Gap}}/{(m+1)^2}.$$
Therefore, for a general dimension $n>4$ and $n\,(\mathbf{mod}\, 4) \not\equiv0$, and fix an arbitrary set of valid linear cuts
$\ell_1^{(n)},\ldots,\ell_m^{(n)}$, we can set $k = \lfloor n/4\rfloor$, and padding the above $=4k$-dimensional certificate-instance pair by zeros will give an $\Omega({n}/{(m+1)^2})$ relaxation gap:
\begin{equation}
\label{eq:padding}
\begin{aligned}
&\Big(\bar x^{(n)},\bar X^{(n)}\Big) := \Big( \Big(\begin{aligned}
\bar x^{(4k)}\\[-1.5mm]
0\,\,\,\,\,\end{aligned}\Big), \Big(\begin{aligned}
\bar X^{(4k)}\,\,\, & 0\,\\[-1.5mm]
0\quad & 0\,
\end{aligned}\Big)\Big),\\
&\Big(\bar Q^{(n)},\bar c^{(n)}\Big) := \Big(\Big(\begin{aligned}
\bar Q^{(4k)}\,\,\, & 0\,\\[-1.5mm]
0\quad & 0\,
\end{aligned}\Big),  
\Big(\begin{aligned}
\bar c^{(4k)}\\[-1.5mm]
\mathbf{1}\,\,\,\,\,\end{aligned}\Big)\Big),
\end{aligned}
\end{equation}
we only to notice that the $4k$-dimensional certificate-instance pair is selected w.r.t. the specific set of cuts: $\ell_j^{(4k)}\big(x,X\big)\! := \!\ell_j^{(n)} \!\Big(\! (x,0),\! \Big(\begin{aligned}
X\,\, & 0\,\\[-1.5mm]
0\,\, & 0\,
\end{aligned}\Big)\!\Big),$ where $(x,X)\in\mathbb{R}^{4k}\times\mathbb{S}^{4k}$ are $4k$-dimensional lifted variables.

\begin{theorem}
\label{thm:Omega(n/m^2)}
For any dimension $n\geq4$ and any $m$ valid linear cuts, there exists a normalized submodular BCQP instance $(Q,c)\in\mathcal{N}^n_{\infty}$ such that the SDP relaxation gap is $\Omega(n/(m+1)^2)$.    
\end{theorem}

\noindent\textbf{BQP duplication.} Note that Proposition \ref{prop:BQP-certificate} provides an $m$-independent relaxation gap of $1/100$ for SDP relaxations with only BQP valid cuts. In this part, we utilize this property to construct a sharper $\Omega(n)$ relaxation gap for this case. Given the $4$-dimensional certificate $\big(x^{\rm BQP}, X^{\rm BQP}\big)$ from Proposition \ref{prop:BQP-certificate}, let us keep the $k$-fold duplication $\bar x^{(4k)} = \mathbf{1}_k\otimes x^{\rm BQP}$ while modifying
\begin{equation}
\label{eq:BQP-Duplicate}
\bar X^{(4k)} = I_{k}\otimes X^{\rm BQP} + J_k\otimes(x^{\rm BQP}\otimes(x^{\rm BQP})^\top), 
\end{equation}
with $J_k:=\mathbf{1}_{k\times k}-I_k$. A remark is that, we adopt this BQP-specific duplication because the general duplication \eqref{eq:Gen-Duplicate} cannot preserve BQP membership of the duplicated certificate. Reversely, \eqref{eq:BQP-Duplicate} is nonlinear and hence do not apply to the derivation of Theorem \ref{thm:Omega(n/m^2)} due to feasibility issue. Let us select the $4$-dimensional certificate $\big(x^{(4)},X^{(4)}\big) = \big(x^{\rm BQP}, X^{\rm BQP}\big)$ from Proposition \ref{prop:BQP-certificate}, then the probabilistic representation of $\mathcal{B}_n$ indicates that there is a random vector $Z\in\{0,1\}^4$ such that $\mathbb{E}[Z]=x^{\rm BQP}$ and $ \mathbb{E}[Z_iZ_j]=X^{\rm BQP}_{ij}$ for $\forall\, 1\leq i<j\leq 4.$ Therefore, taking $k$ independent copies $Z^1,\ldots,Z^k$ yields exactly the duplicated moments in
\eqref{eq:BQP-Duplicate} and hence $\big(\bar x^{(4k)},\big(\bar X^{(4k)}_{ij}\big)_{i<j}\big)\in \mathcal{B}_{4k}.$ Therefore, in dimension $4k$, consider the instance $(\bar Q^{(4k)}, \bar c^{(4k)})$ constructed from \eqref{eq:QC-duplicate} with base instance $(\bar Q^{(4)}, \bar c^{(4)})$ from the parameterized family \eqref{eq:bad-family} with $a=b=1/2$, we have 
$$p^{(4k)}-L_{4k}\big(\bar{x}^{(4k)}, \bar{X}^{(4k)}\big) = k\cdot \left(p^{(4)}-L_4\big(\bar{x}^{(4)}, \bar{X}^{(4)}\big)\right) =k/100.$$
Then applying the padding technique again gives the following theorem.

\begin{theorem}
\label{thm:Omega(n)-BQP-GP}
For any $n\geq4$, there exists a normalized submodular BCQP instance $(Q,c)\in\mathcal{N}^n_{\infty}$ such that the SDP relaxation gap, if only with BQP valid cuts, is at least $\lfloor n/4\rfloor/100 = \Omega(n)$.
\end{theorem}

A short remark is that the BQP duplication \eqref{eq:BQP-Duplicate} does not change the feasibility to diagonal constraints $X_{ii}\leq x_i, i\in[n]$. As the base certificate $\big(x^{\rm BQP},X^{\rm BQP}\big)$ is feasible to this inequality in dimension $n=4$, adding these simple valid cuts to the SDP relaxation with BQP cuts does not deteriorate the $\Omega(n)$ lower bound to $\Omega(n/m^2), m\geq n$ even though they are a set of $n$ general cuts of form \eqref{eq:general-cuts}.

\section{Roadmap To Construct the Parameterized Instance Family }
\label{sec:n4-counterexample}

\subsection{A high-level geometric insights of SDP tightness }
\label{subsec:geometric-view}
First, let us discuss the underlying geometric insight of our construction, which was stated in the introduction. Fix an arbitrary set of $m\geq0$ valid linear cuts, and denote the feasible region of the SDP relaxation as 
$$\mathcal{R}_n = \left\{(x,X)\in\mathcal{S}_n: \ell_j(x,X)\ge0,  \forall  j \in [m]\right\}.$$
We have the following theorem. 


\begin{theorem}
\label{thm:geometric-insights}
The SDP relaxation \eqref{eq:sdp-strengthen} is tight for the submodular BCQP \eqref{eq:boxqp} over all instances $(Q,c)\in\mathcal{N}_\infty^n$ if and only if $\mathcal{R}_n\subseteq\mathcal{M}_n+\mathcal{D}_n$, where $\mathcal{D}_n$ is defined by \eqref{eq:D-n}.
\end{theorem} 
\begin{proof}  As all linear cuts are valid on $\mathcal{M}_n$, then $\mathcal M_n \subseteq \mathcal R_n$, and we have
\[\min_{(x,X)\in\mathcal R_n}
 L(x,X)
\leq 
\min_{(\tilde x,\tilde X)\in\mathcal M_n}
L(\tilde x,\tilde X).\]
Now suppose $\mathcal R_n \subseteq \mathcal M_n+\mathcal D_n$. For   $\forall (x,X)\in\mathcal R_n$, we can write $(x,X)=(\tilde x,\tilde X)+(0,H),$ for some $(\tilde x,\tilde X)\in\mathcal M_n$ and $(0,H)\in\mathcal D_n.$ Fix an arbitrary instance $(Q,c)\in\mathcal N_\infty^n$, denote 
$$L(x,X) = Q\bullet X + c^\top x$$ the objective function of SDP relaxation. Then $Q_{ij}\leq0$, $H_{ij}\leq0$ for $i\neq j$, and
$H_{ii}=0$ indicates that  
$$L(x,X) = L(\tilde x,\tilde X) + Q\bullet H \geq L(\tilde x,\tilde X)\geq \min \left\{L(z,Z) : (z,Z)\in\mathcal{M}_n\right\},$$
which further gives
\[\min_{(x,X)\in\mathcal R_n}
 L(x,X)
\geq 
\min_{(\tilde x,\tilde X)\in\mathcal M_n}
L(\tilde x,\tilde X).\]
Therefore, the equality holds. Also notice that the minimum value of the linear function $L$ over the lifted moment hull $\mathcal{M}_n$ equals exact the minimum value of the BCQP instance. Therefore, we know the SDP relaxation is tight if $\mathcal{R}_n\subseteq \mathcal{M}_n+\mathcal{D}_n$ holds. 

Conversely, if there exists
$(\bar x,\bar X)\in
\mathcal R_n\setminus(\mathcal M_n+\mathcal D_n)$. As $\mathcal M_n+\mathcal D_n$ is closed and convex while $(\bar x,\bar X)$ is a single point, then the
separation hyperplane theorem gives some $(\bar Q,\bar c)$ such that
\[
\bar Q\bullet\bar X+ \bar c^\top\bar{x}
<
\inf_{(x,X)\in\mathcal M_n+\mathcal D_n}
\{\bar Q\bullet X+\bar c^\top x\}.
\]
As the right hand side is bounded below, the separator $\bar Q$ must satisfy $\bar Q\bullet H\ge0$ for $\forall (0,H)\in\mathcal D_n$, 
which is equivalent to $\bar Q_{ij}\le0$ for every $i\neq j$. Therefore, for this specific choice of $(\bar Q,\bar c)$, the above infimum over $\mathcal{M}_n + \mathcal{D}_n$ can be reduced to the minimum over only $\mathcal{M}_n$ and hence equals the BCQP optimal value. With proper rescaling, we can assume
$(\bar Q,\bar c)\in\mathcal N_\infty^n$. Thus a strictly positive SDP relaxation gap exists for the instance $(\bar Q,\bar c)$, while $(\bar x,\bar X)$ can serve as a certificate.  
\end{proof}

This theorem suggests that whether the SDP relaxation is tight reduces to whether the incorporated linear cuts can cut off the whole region  $\mathcal{S}_n\backslash(\mathcal{M}_n+\mathcal{D}_n)$. Moreover, instead of randomly sampling the $(Q,c)$ instance, the proof suggests that a more reasonable approach could be searching a certificate  $(\bar{x},\bar{X})$ first and then construct the QP instance by finding a separation hyperplane.


\subsection{Numerical approach for instance construction}
\label{subsec:construction}
First, let us fix the linear cuts of $\mathcal{R}_n$ as the McCormick inequalities. Then we propose to construct the random candidate certificates by solving SDPs with randomly generated coefficients $G$ and $g$: 
$$(\tilde x,\tilde X) = \argmin_{(x,X)\in\mathcal{R}_n} \,\,G\bullet X + g^\top x.$$
These solutions are boundary extreme points of $\mathcal{R}_n$, and are  more likely to stay out of $\mathcal{M}_n+\mathcal{D}_n$. It is worth noting that the randomly generated $(G,g)$ does not necessarily give a submodular BCQP. Given this candidate certificate, we form the following separation problem 
\begin{equation}
\label{eq:separation}
\begin{aligned}
    & \min_{(Q,c)\in\mathcal{N}_\infty^n} \,\,\, Q \bullet \tilde X + c^\top \tilde x\\
    & \quad\,\,  \mathrm{s.t.} \quad\,\,\, z^{\!\top}\!  Qz \!+\! c^{\!\top}\! z\!\geq\! 0,\,\,\,\, \forall z\in[0,1]^n\!. 
\end{aligned}
\end{equation}
Then as long as we find a $(Q,c)$ with a strictly negative objective value, then it will be a non-tight instance for the SDP relaxation with McCormick inequalities. 

A remark is that one can indeed solve a more general separation problem without requiring the QP objective value to be nonnegative. This will introduce a max-min problem 
\[\max_{(Q,c)\in\mathcal{N}_\infty^n}\,\min_{z\in[0,1]^n} \Big(z^\top\! Qz+c^\top z\Big)-\Big(Q \!\bullet\! \tilde X \!+  c^{\!\top}\! \tilde x\Big) \]
for which the inner minimization can be solved by active face enumeration when $n$ is small.
Nevertheless, the current formulation \eqref{eq:separation} is already sufficient for finding a non-tight instance. 

Note that the separation problem \eqref{eq:separation} is a standard semi-infinite linear programming. In dimension $n=4$ where BCQP can be solved by brute force active face enumeration, then we can  solve it by the standard cutting-plane method \cite{hettich1993semi}. This approach can efficiently provide us an instance-certificate pair $(\bar x,\bar X,\bar c,\bar Q)$, but the resulting gap is often moderate, usually in  $10^{-5}\sim10^{-4}$ order. Then starting from this initial point, we perform an alternating block  minimization for problem
\begin{equation}
\label{eq:alternating}
\begin{aligned}
    & \min_{(Q,c)\in\mathcal{N}_\infty^n} \,\, \min_{(x,X)\in\mathcal{R}_n} \,\,\, Q \bullet X + c^\top  x\\
    & \quad\,\,  \mathrm{s.t.} \quad\,\,\, z^{\!\top}\!  Qz \!+\! c^{\!\top}\! z\!\geq\! 0,\,\,\,\, \forall z\in[0,1]^n,
\end{aligned}
\end{equation}
which further allows a minimization over the certificate point compared to \eqref{eq:separation}. Interestingly, the sequence of iterates for this alternating minimization scheme, which are also valid non-tightness certificate-instance pairs, starts to exhibit the interesting pattern of \eqref{eq:bad-family} with $a,b\approx1/2$. This leads us to the guess of the general parameterized family $(Q(a,b), c(a,b))$.

\section{Conclusion}  We study the worst-case relaxation gap of SDP relaxations for submodular BCQP. In dimension $n\geq4$, we establish an \(\Omega(n)\) gap for SDP relaxations with BQP-valid cuts and an \(\Omega(n/m^2)\) gap for \(m\) arbitrary valid linear cuts. On the one hand, this shows that the set $\mathcal{M}_n+\mathcal{D}_n$ cannot be established by incorporating finitely many valid linear cuts over the lifted space $(x, X)$. On the other hand, it also reveals a fundamental limitation of BQP-valid cuts, including commonly used McCormick, triangle, clique, and cycle inequalities, etc. Therefore how to systematically design non-BQP-valid cuts to reduce the relaxation gap remains an interesting topic for future research.

\section{Code and Data Disclosure}\label{sec:Code and Data Disclosure} 




No external datasets were used in this paper. All numerical computations were based on synthetically generated instances. The instance construction procedure and the corresponding numerical approaches are described in Section~\ref{subsec:construction}.

\bibliographystyle{abbrvnat}
\bibliography{reference}

\appendix
\renewcommand{\theHsection}{appendix.\Alph{section}}

\section{Exact rational certificate for BQP non-tightness result (Proposition \ref{prop:BQP-certificate})}
\label{app:BQP-certificate}
Select $a=b=1/2$, the parameterized family \eqref{eq:bad-family} gives:
\[
Q\left(\frac12,\frac12\right)=
\begin{pmatrix}
\frac14 & -\frac12 & 0 & 0 \\
-\frac12 & 1 & -1 & 0 \\
0 & -1 & 1 & -\frac12 \\
0 & 0 & -\frac12 & \frac14
\end{pmatrix},
\quad
c\left(\frac12,\frac12\right)=
\begin{pmatrix}
\frac12 \\
1 \\
0 \\
0
\end{pmatrix}.
\]
Select six binary points $z^1=(0,0,0,0)$, $z^2=(0,0,0,1)$, $z^3=(0,0,1,1)$, $z^4=(0,1,1,1)$, $z^5=(1,0,1,0)$, $z^6=(1,1,1,1)$. Define their lifted points (projected to $\mathcal{B}_4$) as $P_s = (z^s,(z_i^sz_j^s)_{i<j})$ for 
$s=1,\cdots,6$. We do not select other vertices of $\mathcal{B}_4$ and hence their corresponding coefficient $y_s$ in \eqref{eq:SDP-BQP-certificate} is set to 0. Next we provide the exact value of $x^{\rm BQP}, X^{\rm BQP}, y^{\rm BQP}$ as follows
\begin{align*} 
X^{\rm BQP}=
\begin{pmatrix}
\frac{19}{125} & \frac{9}{50} & \frac{41}{200} & \frac{9}{50} \\[1mm]
\frac{9}{50} & \frac{32}{125} & \frac{26}{75} & \frac{26}{75} \\[1mm]
\frac{41}{200} & \frac{26}{75} & \frac{563}{1000} & \frac{377}{600} \\[1mm]
\frac{9}{50} & \frac{26}{75} & \frac{377}{600} & \frac{371}{500}
\end{pmatrix},\quad  
x^{\rm BQP}=
\begin{pmatrix}
\frac{41}{200} \\[1mm]  \frac{26}{75} \\[1mm] \frac{49}{75} \\[1mm] \frac{159}{200} \end{pmatrix}. 
\end{align*}
Then $(x^{\rm BQP},X^{\rm BQP})\in\mathcal{S}_4$. Select $$y^{\rm BQP}=
\left(\frac{9}{50},\frac{1}{6},\frac{169}{600},\frac{1}{6},\frac{1}{40},\frac{9}{50},0,\cdots,0\right)\in\Delta_{16}.$$
Then it holds that 
$$\Big(x^{\rm BQP},\big(X^{\rm BQP}_{ij}\big)_{i<j}\Big) = y_1^{\rm BQP}\cdot P_1+\cdots+y_6^{\rm BQP}\cdot P_6,$$
and the objective value is 
\[Q\left(\frac12,\frac12\right)\bullet X^{\rm BQP} + c\left(\frac12,\frac12\right)\bullet x^{\rm BQP} = -1/100<0.\]

\section{Exact conic certificate for face-active cuts (Proposition \ref{prop:conic-certificate})}
\label{app:conic-certificate}
The certificate $(\bar{x},\bar{X})$ are given as follows:
\[
\bar{x}=\begin{pmatrix}
\frac{1}{6}\\[1mm]\frac{1}{2}\\[1mm]1\\[1mm]\frac{9}{8}
\end{pmatrix},\qquad
\bar{X}=\begin{pmatrix}
\frac{7}{72} & \frac{1}{8} & \frac{1}{6} & \frac{1}{8}\\[1mm]
\frac{1}{8} & \frac{7}{24} & \frac{1}{2} & \frac{1}{2}\\[1mm]
\frac{1}{6} & \frac{1}{2} & 1 & \frac{9}{8}\\[1mm]
\frac{1}{8} & \frac{1}{2} & \frac{9}{8} & \frac{87}{64}
\end{pmatrix}.
\]
Direct computation gives $(\bar{x},\bar{X})\in\mathcal{S}_4$. Select the grid points as
$$\xi^1=(0,{1}/{3},1,1), \qquad  \xi^2=({2}/{3},1,1,1),\qquad  \xi^3=({1}/{3},0,1,0).$$
And set the conic combination coefficients as 
$$\bar\lambda_1 = \frac{15}{16},\qquad \bar\lambda_2 = \frac{3}{16},\qquad\bar\lambda_3 = \frac{1}{8}.$$
Then we have 
\begin{equation}
\label{eq:conic-combination-Psi}
\Psi(\bar{x},\bar{X}) = \bar\lambda_1\psi(\xi^1) + \bar\lambda_2\psi(\xi^2)+\bar\lambda_3\psi(\xi^3)
\end{equation}
ensures that $\Psi(\bar{x},\bar{X}) \in \mathcal{K} \subsetneq \text{cl}(\mathcal{K})$. 
Finally, substituting $\bar{x}$ and $\bar{X}$ to the SDP objective gives:
$$L_\eta(\bar{x},\bar{X})-f_\eta(x^*) = - \frac{25}{2304} + \frac{31}{64}\eta<-\frac{1}{200}, \quad \forall \eta\leq\frac{1}{100}.$$  

\section{Proof of Lemma \ref{lemma:perturb-optimality} }
\label{app:perturb-optimality}
\begin{proof}
Note that $x^*(\theta):=(0,0,\theta_1,\theta_2)$, for any feasible $x\in[0,1]^4$, write $d = x-x^*(\theta)$, then we have $d_1,d_2\geq0$. Because $f_{\eta,\theta}$ is a quadratic function, and use the relationship $\theta_1+d_3 = x_3$, we know 
\begin{align*}
    & \Delta_f = d^\top Q^{\eta,\theta}d + d^\top\nabla f_{\eta,\theta}(x^*(\theta))\\
    & \!\!= \frac{d_1}{2}\!+\!(1-2x_3)d_2 \!+\! \left(\frac{d_1}{2}-d_2\right)^2\!\!\! + \left(d_3-\frac{d_4}{2}\right)^2 \!\!+ \eta d_4^2,
\end{align*} 
where we write $\Delta_f := f_{\eta,\theta}(x)-f_{\eta,\theta}(x^*(\theta))$ to simplify the notation. When $x_3\leq1/2$, every term in the above difference is nonnegative, and only when $d=0$, it achieves the minimum $0$. On the other hand, when $x_3>1/2$, the above difference can be rewritten as 
\begin{align*}
\Delta_f = \left(\frac{d_1}{2}-d_2 + x_3 - \frac12\right)^2\!\!\!+d_1(1-x_3) + H_\theta(x_3,x_4) 
\end{align*}
where $H_\theta(x_3,x_4)$ is quadratic function of $x_3,x_4$ and $\theta$. We do not really need to write down its general form as it is too lengthy. Instead, we only need to notice that at $\theta = \theta^* = (1/4,1/2)$, we have 
\begin{align*}
&H_{\theta^*}(x_3,x_4)=\left(x_3-\frac{x_4}{2}\right)^2\!\!-\left(x_3-\frac12\right)^2\!\!+\eta\left(x_4-\frac12\right)^2 
\end{align*}  
and its minimum value on $(x_3,x_4)\in[1/2,1]\times[0,1]$ is $\frac{\eta}{4(1+4\eta)}$. Consequently, when $x_3\geq1/2$, we have $\Delta_f\geq H_{\theta^*}(x_3,x_4) \geq\frac{\eta}{4(1+4\eta)}>0$, strictly. Hence, due to the continuity of the optimal value in terms of $\theta$. There exists an open neighborhood $\Theta^\eta\ni\theta^*$ such that $\Delta_f\geq H_{\theta}(x_3,x_4)>0$ on $[1/2,1] \times [0,1]$. Therefore, for $\theta$ in this neighborhood, $\Delta_f\geq0$ for all $x\in[0,1]^4$ and hence $x^*(\theta)$ is the optimal solution to perturbed function $f_{\eta,\theta}$. 

Finally, as $ L_{\eta,\theta}(\bar{x},\bar{X})-f_{\eta,\theta}(x^*(\theta))$ is a continuous function of $\theta$ and $L_{\eta,\theta^*}(\bar{x},\bar{X})-f_{\eta,\theta^*}(x^*(\theta^*))<-1/200$, $\Theta^\eta$ can always be chosen small enough such that $L_{\eta,\theta}(\bar{x},\bar{X})-f_{\eta,\theta}(x^*(\theta))<-1/250$. \end{proof} 

\section{Proof of Inequality \eqref{eq:Certificate-lower-bound}}
\label{app:certificate-lower-bound}

\begin{proof} Consider an arbitrary valid cut $\ell(x,X)$  with an induced quadratic function $q_\ell(\theta)$ in the forms of \eqref{eq:general-cuts} and \eqref{eq:q-theta}, respectively. Let us inherit the notations of $\alpha$, $\psi$, and $\Psi$ from the discussion of Lemma \ref{lemma:dual-cone-alpha}, we can partition the monomials of $\ell$, cf.  \eqref{eq:general-cuts}, into a summation of two terms: 
$$\ell(x,X) = \alpha^\top\Psi(x,X) + R(x,X),$$
where 
$$R(x,X) = c + b_3x_3+b_4x_4+d_3X_{33}+h_{34}X_{34}+d_4X_{44}.$$ 
This time, as $\ell$ is not face-active on $F$, the coefficient $\alpha$ need not necessarily stay in the dual cone $\mathcal{K}^*$. As $R$ does not contain $x_1,x_2$ and the corresponding lifted variables, for any vector $x\in\mathbb{R}^4$, we have $R(x,xx^\top) = q_\ell(x_{3:4})$ where $x_{3:4}$ contains the last two coordinates of $x$. Hence, for the certificate $(\bar x, \bar X, \bar \lambda)$ in Appendix \ref{app:conic-certificate},    \eqref{eq:conic-combination-Psi} gives 
$$\alpha^\top  \Psi(\bar{x},\bar{X})  = \sum_{i=1}^3 \bar\lambda_i\alpha^\top\psi(\xi^i) = \sum_{i=1}^3 \bar\lambda_i\cdot\left(\ell(\xi^i,\xi^i(\xi^i)^\top) - q_\ell(\xi^i_{3:4})\right)\geq -\sum_{i=1}^3 \bar\lambda_i q_\ell(\xi^i_{3:4}).$$
where we use the fact that $\psi(x) = \Psi(x,xx^\top)$ for all $x$, $\bar\lambda_i\geq0$, and $\ell(\xi^i, \xi^i(\xi^i)^\top) \geq 0$. Consequently, 
$$\ell(\bar x,\bar X) =  \alpha^\top\Psi(\bar x,\bar X) + R(\bar x,\bar X) \geq R(\bar x,\bar X)-\sum_{i=1}^3 \bar\lambda_i q_\ell(\xi^i_{3:4}).$$ Then substituting the detailed values of $\bar x, \bar X, \bar\lambda, \xi^i$ in Appendix \ref{app:conic-certificate} to this inequality gives 

\[
\begin{aligned}
&\ell(\bar{x},\bar{X})
 \geq 
-\frac{1}{4}\left(c+b_3+d_3\right)
+ \frac{15}{64}d_4\\
&\quad\, = -\frac{q_\ell(1,0)}{4} \!+\! \frac{15}{32}\!\left(\!q_\ell(0,\!1) \!-\! 2q_\ell\Big(0,\!\frac12\Big) \!+\! q_\ell(0,0)\!\right)\\
&\quad\, \overset{(i)}{\geq} -\frac{19}{16}\cdot\frac{\|q_\ell\|_{[0,1]^2}}{\|q_\ell\|_{D}}\cdot \|q_\ell\|_{D}\\
& \quad\,  \overset{(ii)}{\geq}  - C_D\|q_\ell\|_{D}.
\end{aligned}
\]
Here, (i) is by $q_\ell(0,1)\geq0, q_\ell(0,0)\geq0$, and the norm $$\|q_\ell\|_{[0,1]^2}:=\max\{q(\theta):\theta\in[0,1]^2\}$$ is defined analogously to $\|\cdot\|_D$. And (ii) must hold for some constant $C_D>0$ that relies only on $D$ due to the equivalence of norms on the six-dimensional linear space of bivariate quadratic functions. 
\end{proof}

\end{document}